\documentclass[11pt]{article}

\usepackage[latin1]{inputenc}
\usepackage{t1enc}
\usepackage{amsmath, amssymb, amsfonts, amsthm, amsopn,epsfig}
\usepackage[none]{hyphenat}
\usepackage{tikz}
\usepackage{float}
\usepackage{graphicx}
\usepackage{caption}
\usepackage[margin=1in]{geometry}
\usepackage[toc,page]{appendix}
\usepackage{epstopdf}
\usepackage{etoolbox}
\usepackage[colorlinks=true]{hyperref}
\usepackage{dsfont}
\hypersetup{
	colorlinks=true,
	linkcolor=blue,
	filecolor=magenta,      
	urlcolor=cyan,
	citecolor=blue
}
\usepackage{cleveref}
\usepackage{tikz}

\usetikzlibrary{positioning,arrows,shapes,decorations.markings,decorations.pathreplacing,matrix}
\tikzstyle{vertex}=[circle,draw=black,fill=black,inner sep=0,minimum size=5pt,text=white,font=\footnotesize]
\usetikzlibrary{calc}

\theoremstyle{plain}
\newtheorem{theorem}{Theorem}[section]

\newtheorem{claim}[theorem]{Claim}
\newtheorem{lemma}[theorem]{Lemma}

\theoremstyle{definition}
\newtheorem{definition}{Definition}

\title{Exponential Erd\H{o}s-Szekeres theorem for matrices}
\author{Recep Altar \c{C}i\c{c}eksiz\thanks{Ume\r{a} University, \emph{e-mail}: \textbf{altar.ciceksiz@umu.se}},
Zhihan Jin\thanks{ETH Zurich, \emph{e-mail}: \textbf{zhihan.jin@ifor.math.ethz.ch }},
Eero R\"aty\thanks{Ume\r{a} University, \emph{e-mail}: \textbf{eero.raty@umu.se}},
Istv\'an Tomon\thanks{Ume\r{a} University, \emph{e-mail}: \textbf{istvan.tomon@umu.se}}}
\date{}

\begin{document}
	\sloppy 
	
	\maketitle

\begin{abstract}
    In 1993, Fishburn and Graham established the following qualitative extension of the classical Erd\H{o}s-Szekeres theorem. If $N$ is sufficiently large with respect to $n$, then any $N\times N$ real matrix contains an $n\times n$ submatrix in which every row and every column is monotone. We prove that the smallest such $N$ is at most $2^{n^{4+o(1)}}$, greatly improving the previously best known double-exponential upper bound, and getting close to the best known lower bound $n^{n/2}$.

    In particular, we prove the following surprising sharp transition in the asymmetric setting. On one hand, every $8n^2\times 2^{n^{4+o(1)}}$ matrix contains an $n\times n$ submatrix, in which every row is mononote. On the other hand, there exist $n^{2}/6\times 2^{2^{n^{1-o(1)}}}$  matrices containing no such submatrix .  
\end{abstract}

\section{Introduction}

The Erd\H{o}s-Szekeres theorem \cite{ESz35} from 1935 is one of the cornerstone results of Ramsey theory, with countless applications in analysis, combinatorics, geometry and logic. It states that any sequence of $(n-1)^2+1$ real numbers contains a monotone increasing or decreasing subsequence of length $n$, and this bound is the best possible. See Steele \cite{Steele} for several different proofs and applications.

Since then, many generalizations and extensions of the Erd\H{o}s-Szekeres theorem are proposed  \cite{BM1,BM2,FG93,K73,K53,LS18,M69,SzT01}, among which one of the most natural is due to Fishburn and Graham \cite{FG93}. A \emph{$d$-dimensional array} is a function $A:S_1\times\dots\times S_d\rightarrow \mathbb{R}$, where $S_1,\dots,S_d$ are finite subsets of integers. Say that $A$ is \emph{monotone} if $f(x)=A(a_1,\dots,a_{i-1},x,a_{i+1},\dots,a_d)$ is a monotone function for every fixed $a_1,\dots,a_{i-1},a_{i+1},\dots,a_{d}$, and whether it is increasing or decreasing only depends on $i$. With this notation in our hand, the result of Fishburn and Graham states that for every $d$ and $n$, there exists a smallest number $N=M_d(n)$ such that every $d$-dimensional array of size $N\times\dots\times N$ contains an $n\times\dots\times n$ sized $d$-dimensional monotone subarray. Observe that the Erd\H{o}s-Szekeres theorem is equivalent to the statement that $M_1(n)=(n-1)^2+1$. The main problem we are interested in is finding the order of magnitude of $M_d(n)$ for $d\geq 2$.

The first proofs of Fishburn and Graham \cite{FG93} of the existence of $M_d(n)$ gave Ackermann-type upper bounds of order $d$ for $d\geq 4$, while in case $d=2$, they found that $M_2(n)\leq \mbox{tw}_5(O(n))$. Here, the \emph{tower function} is defined as $\mbox{tw}_1(x):=x$ and $\mbox{tw}_k(x):=2^{\mbox{tw}_{k-1}(x)}$. On the other hand, the best known lower bound for every $d\geq 2$ is $$M_d(n)\geq n^{(1-1/d)n^{d-1}},$$ due to a simple probabilistic argument. Recently, the upper bounds have been greatly improved in every dimension. Buci\'c, Sudakov, and Tran \cite{BST} showed that 
$$M_2(n)\leq 2^{2^{O(n)}}, M_3(n)\leq 2^{2^{O(n^2)}}, \mbox{ and } M_d(n)\leq \mbox{tw}_4(O_d(n^{d-1})) \mbox{ for } d\geq 4.$$   
Lychev \cite{L21} slightly improved the upper bound for $d=2$, but the improved bound was still of the order $2^{2^{O(n)}}$. Gir\~ao, Kronenberg, and Scott \cite{GKS} removed one exponential for $d\geq 4$, and established the inequality $M_d(n)\leq 2^{2^{O_d(n^{d-1})}}$, which then holds uniformly for every $d\geq 2$. They proved this bound by considering a more general problem about the Ramsey properties of the Cartesian products of graphs. Despite the recent progress, there is still an exponential gap between the best known lower and upper bound for every $d\geq 2$. Among these cases, perhaps the most puzzling was whether $M_2(n)$ grows exponentially or double-exponentially. Our main theorem answers this question.

\begin{theorem}\label{thm:main}
    There exists a constant $c>0$ such that $M_2(n)\leq 2^{cn^{4}(\log n)^2}$.
\end{theorem}

 In what follows, we consider 2-dimensional arrays, which we refer to simply as matrices. Say that a matrix is \emph{row-monotone} if every row is monotone increasing or every row is monotone decreasing. In \cite{BST}, in order to prove a double-exponential upper bound on $M_2(n)$, a key idea is to show that any $2n\times N$ matrix contains an $n\times n$ row-monotone submatrix if $N>(n-1)^{2^{2n}}$. A natural idea would be to show that this bound on $N$ can be significantly improved. Unfortunately, this is not possible due to a construction of Burkill and Mirsky \cite{BM1}, see also Lichev \cite{L21}.

 Our key contribution is the following theorem, which shows that by slightly increasing the number of rows, we can indeed decrease the number of columns significantly in order to find an $n\times n$ row-monotone submatrix.

\begin{theorem}\label{thm:main_upperbound}
There exists $c>0$ such that if $N\geq 2^{cn^{4}(\log_2 n)^2}$, then every $8n^2\times N$ matrix contains a row-monotone $n\times n$ matrix. 
\end{theorem}

This theorem tells us that roughly $n^2$ rows are enough to guarantee an $n\times n$ row-monotone submatrix if the number of columns is exponential. It comes as a surprise that there is a very sharp transition for this phenomenon around $n^2$. In the next theorem, we prove that if the number of rows is slightly less than $n^{2}$, then we still need a double-exponential number of columns.

\begin{theorem}\label{thm:main_lowerbound}
    For every sufficiently large $n$, there exists an $\lfloor n^{2}/6\rfloor \times 2^{2^{\lfloor n/2\log_2 n\rfloor }}$  matrix with no $n\times n$ row-monotone submatrix.
\end{theorem}

Finally, let us mention that our main theorem has further implications about \emph{lexicographic arrays} as well. A $d$-dimensional array $A$ is \emph{lex-monotone} if there exists a permutation $\sigma\in S_d$ and a sign vector $s\in \{-1,1\}^{d}$ such that
\begin{align*}
    A(a_1,\dots,a_d)&<A(b_1,\dots,b_d)\Leftrightarrow \\
    (s(\sigma(1))\cdot a_{\sigma(1)},\dots,s(\sigma(d))\cdot a_{\sigma(d)})&<_{LEX}(s(\sigma(1))\cdot b_{\sigma(1)},\dots,s(\sigma(d))\cdot b_{\sigma(d)}).
\end{align*}
Here $<_{LEX}$ denotes the lexicographic ordering, that is, $(x_1,\dots,x_d)<_{LEX} (y_1,\dots,y_d)$ if $x_b<y_b$, where $b$ is the smallest index such that $x_b\neq y_b$. Fishburn and Graham \cite{FG93} proved that there exists a smallest $L_d(n)=N$ such that every $d$-dimensional $N\times\dots\times N$ array contains a $d$-dimensional $n\times\dots\times n$ lex-monotone subarray. This result has found applications in poset dimension theory \cite{FFT99} and computational complexity theory \cite{BK10}. 

As every lex-monotone array is also monotone, we trivially have $M_d(n)\leq L_d(n)$. On the other hand, Fishburn and Graham \cite{FG93} proved that $L_2(n)\leq M_2(2n^2-5n+4)$, and Buci\'c, Sudakov, and Tran \cite{BST} established $L_d(n)\leq M_d(2^{O_d(n^{d-2})})$ for $d\geq 3$. Combining this with the result of Gir\~ao, Kronenberg, and Scott \cite{GKS}, we get the best known upper bound $L_d(n)\leq \mbox{tw}_4(O_d(n^{d-2}))$ for $d\geq 3$. However, for $d=2$, Theorem \ref{thm:main} immediately implies the following improvement, which also has the right order on an exponential scale.

\begin{theorem}
    There exists a constant $c>0$ such that $L_2(n)\leq 2^{cn^8(\log n)^2}$.
\end{theorem}

Our paper is organized as follows. In the next section, we introduce our notation and some basic results. Then, in Section \ref{sect:1}, we present our construction for Theorem \ref{thm:main_lowerbound}. Then, in Section \ref{sect:2}, we prove Theorem \ref{thm:main_upperbound}, and we finish our paper with the proof of Theorem \ref{thm:main} in Section \ref{sect:3}.

 \section{Preliminaries}

 We omit the use of floors and ceilings whenever they are not crucial.

 \subsection{Rooted trees}

 Let us introduce some standard (and also some less standard) terminology about \emph{rooted trees}. A \emph{rooted tree} is a pair $(T,r)$, where $T$ is a tree and $r$ is a vertex of $T$, called the \emph{root}. For ease of notation, we write simply $T$ instead of $(T,r)$. The \emph{depth} of a vertex in $T$ is its distance from the root. A vertex $v$ is a \emph{descendant} of a vertex $w$ if $w$ is contained in the unique path connecting $v$ and $r$. In this case, we also say $w$ is an \emph{ancestor} of $v$ (unconventionally, we say that $v$ is both a descendant and an ancestor of itself). Say that $v$ and $w$ are \emph{related} if one is an ancestor of the other. The \emph{children} of a vertex $v$ are the neighbours of $v$ that are also descendants. 

 Given a set of vertices $S\subset V(T)$, their \emph{common ancestor} is the vertex  of largest depth that is an ancestor of every element of $S$. We denote the common ancestor by $\delta_T(S)=\delta(S)$, and in case $S=\{x,y\}$, we write simply $\delta(x,y)$. Note that for any $S\subset V(T)$, $\delta(S)=\delta(x,y)$ for some $x,y\in S$.
 
 Given a set of vertices, it naturally induces a rooted subtree in $T$ as follows (not to confuse with the usual graph theoretic notion of induced subgraph).

 \begin{definition}[Induced subtree]\label{def:inducedtree}
 Given a rooted tree $T$ and $X\subset V(T)$, we define the rooted tree $T[X]$ as follows. The vertices of $T[X]$ are the common ancestors $\delta(v,w)$ for every $v,w\in X$, the root of $T[X]$ is $\delta(X)$, and distinct vertices $x$ and $y$ are joined by an edge if one is an ancestor of the other in $T$, and there is no $z\in V(T[X])\setminus \{x,y\}$ that is the descendant of one, and the ancestor of the other. We refer to $T[X]$ as the \emph{subtree of $T$ induced by $X$}. 
 \end{definition}

 Let us highlight some simple, yet important properties of $T[X]$. The common ancestor function $\delta_{T[X]}$ agrees with $\delta_T$ on the vertices of $T[X]$. Also, given $Y\subset X$, we have $(T[X])[Y]=T[Y]$.

 \begin{claim}\label{claim:union}
     Let $X,Y\subset V(T)$ and assume that the root $x$ of $T[X]$ and root $y$ of $T[Y]$ are not related. Then $T[X\cup Y]$ is the disjoint union of $T[X]$ and $T[Y]$, together with its root $\delta(x,y)$ that is joined to $x$ and $y$.
 \end{claim}

 \begin{proof}
     It is enough to show that for $a\in X$ and $b\in Y$, we have $\delta(a,b)=\delta(x,y)$. Note that $\delta(a,b)$ is related to $x$ as it is an ancestor of $a$. But $\delta(a,b)$ is not a descendant of $x$, as no descendant of $x$ is related to $b$. Therefore, $\delta(a,b)$ is an ancestor of $x$. Similarly, $\delta(a,b)$ is an ancestor of $y$. But then we must have $\delta(a,b)=\delta(x,y)$. 
 \end{proof}

 \medskip

 For $m\in \mathbb{N}$, we will denote by $BT_m$ the \emph{perfect rooted binary tree} of height $m$. That is, $BM_0$ is a single vertex, which is also the root, and  $BT_m$ is constructed by attaching two children to every leaf of $BT_{m-1}$. We will identify the leaves of $BT_m$ with the numbers $1,\dots,2^m$ from left to right. More precisely, if a leaf is identified with $i$ in $BT_{m-1}$, then its two children are identified with $2i-1$ and $2i$ for $i=1,\dots,2^{m-1}$. With the term binary tree, we always refer to a tree that is a perfect binary tree $BT_m$ for some $m$.

 Given two sets of real numbers $A$ and $B$, write $A<B$ if every element of $A$ is less than every element of $B$.

 \begin{claim}\label{claim:distinct_root}
     Let $A,B\subset [2^{m}]$ be two sets of leaves in $BT_m$ such that $A<B$. Then the root of $BT_m[A]$ is distinct from the root of $BT_m[B]$.
 \end{claim}

 \begin{proof}
     If either $A$ or $B$ has only one element, the claim is trivial, so suppose that $|A|,|B|\geq 2$. Assume that the roots of $BT_m[A]$ and $BT_m[B]$ coincide, and let us denote it by $q$. Note that $q$ is not a leaf, so it has two children, $\ell$ (left) and $r$ (right). Let $L$ be the set of leaves that are  descendants of $\ell$ and let $R$ be the set of leaves that are descendants of $r$. Then $L$ and $R$ are disjoint intervals such that $L<R$. The set $A$ is not contained in either $L$ or $R$, otherwise $\ell$ or $r$ is an ancestor of every element of $A$. Similarly, $B$ is not contained in either $L$ or $R$. But $A, B\subset L\cup R$, so we cannot have $A<B$, which is a contradiction. 
 \end{proof}

 We  highlight that if $A,B$ are two sets of leaves in $BT_m$ such that $BT_m[A]$ and $BT_m[B]$ are binary trees of height $t$, and their roots are not related, then $BT_m[A\cup B]$ is a binary tree of height $t+1$ by Claim \ref{claim:union}.

The final simple claim we need is that in each binary tree, we can find an induced binary tree which only contains vertices from a specified set of depths.

 \begin{claim}\label{claim:levels}
  Let $Z\subset \{0\dots,m-1\}$. Then there exists $Q\subset [2^m]$ such that $BT_m[Q]$ is a binary tree of height $|Z|$, and the depth of each non-leaf vertex of $BT_m[Q]$ with respect to $BT_m$ is contained in $Z$.
\end{claim}

\begin{proof}
    The set $$Q=\left\{1+\sum_{z\in Z} s_z\cdot 2^{m-1-z}:\forall z\in Z, s_z\in \{0,1\}\right\}$$
    suffices. Indeed, if $a$ and $b$ are leaves, and  the first digit in which the $m$ digit binary expansions of $a-1$ and $b-1$ differ is the $t$-th digit, then $t-1$ is the depth of $\delta(a,b)$.  
\end{proof}

\subsection{Bipartite Ramsey problem}

We will also make use of the following asymmetric variants of the bipartite Ramsey problem. 

\begin{lemma}\label{lemma:zarankiewicz}
    Let $d,t,s,n$ be positive integers such that $t\geq 4s^2$ and $d\geq 4n\cdot 2^s$. Then every $d\times t$ matrix whose entries are colored with red or blue contains an $n\times s$ sized monochromatic submatrix. 
\end{lemma}

\begin{proof}
   First, we show that 
   \begin{equation}\label{equ:1}
   n\binom{t}{s}\leq \binom{t/2}{s}\cdot \frac{d}{2}
   \end{equation} 
   is satisfied. Indeed, using that $t\geq 4s^2$, we can write 
   $$\frac{\binom{t}{s}}{\binom{t/2}{s}}=\frac{t}{t/2}\cdot \frac{t-1}{t/2-1}\dots \frac{t-s+1}{t/2-s+1}\leq 2^s \left(1+\frac{1}{2s}\right)^s\leq 2^{s+1},$$
   so the condition $d\geq 4n\cdot 2^s$ implies (\ref{equ:1}).

    Consider a $d\times t$ matrix whose entries are colored red or blue. By symmetry, we can assume that at least half of the rows contain at least $t/2$ red entries. Assign to each such row at least  $\binom{t/2}{s}$ pieces of $s$-tuples of columns, where each column intersects the row in a red entry. As (1) holds, by the pigeonhole principle we can find an $s$-tuple of columns which is assigned to at least $n$ different rows. The intersection of these $s$ columns and $n$ rows gives the desired submatrix.
\end{proof}

\begin{lemma}\label{lemma:zarankiewicz_lower}
    Let $n$ be a positive integer. Then there exists an $n^{2}/6\times 2^{n/2\log_2 n}$ matrix with entries colored red and blue such that it contains no $n\times \lceil\log_2 n\rceil$ sized monochromatic submatrix.
\end{lemma}

\begin{proof}
    Let $d=n^{2}/6$, $t=2^{n/2\log_2 n}$, and  $s=\lceil\log_2 n\rceil$. Color each entry of a $d\times t$ matrix red or blue with probability $1/2$, independently from each other. Then the expected number of monochromatic $n\times s$ sized submatrices is 
    $$2^{1-ns}\binom{d}{n}\binom{t}{s}\leq 2^{1-ns}\left(\frac{ed}{n}\right)^nt^s.$$
    Taking the base 2 logarithm of the right hand side, we get
    \begin{align*}
        1-ns+n\log_2 \frac{ed}{n}+s\log_2 t&< 1-(n-\log_2 t)s+(-n+n\log_2 n)\\
        &\leq 1-(n-\log_2 t)\log_2 n+(-n+n\log_2 n)\\
        &=1-n+(\log_2 t)\cdot \log_2 n=1-\frac{n}{2}\leq 0.
    \end{align*}
     Here, the first inequality holds by noting that $\frac{ed}{n}<\frac{n}{2}$. Hence, the expectation is less than 1, so there is a coloring with no monochromatic $n\times s$ submatrix.
    \end{proof}

\section{Lower bound for row monotone matrices --- Proof of Theorem \ref{thm:main_lowerbound}}\label{sect:1}

We start with the proof of Theorem \ref{thm:main_lowerbound}, as it also serves as the main inspiration for the proof of Theorem \ref{thm:main_upperbound}. Before we proceed, let us recall the definition of \emph{colexicographic ordering} for binary sequences. Given $x,y\in \{0,1\}^N$ such that $x\neq y$, let $\delta(x,y)$ denote the largest coordinate $b$ such that $x(b)\neq y(b)$ (as we shall see later, it is not a coincidence that we use the same $\delta$ as in the case of common ancestors). Then, we write $x<_{COL} y$ if $y(\delta(x,y))=1$. It is well known that $<_{COL}$ is a total ordering on $\{0,1\}^N$, and it is called the \emph{colexicographic ordering}.

\begin{proof}[Proof of Theorem \ref{thm:main_lowerbound}]
    Let $d=n^{2}/6$, and $t=2^{n/2\log_2 n}$. We define a $d\times 2^t$ sized integer matrix $A$ with no $n\times n$ row-monotone submatrix.
    
    By Lemma \ref{lemma:zarankiewicz_lower}, there exists a $d\times t$ sized matrix with entries $-1$ and $1$ that does not contain an $n\times \lceil\log_2 n\rceil$ sized submatrix with all entries $1$ or all entries $-1$. Let $s_1,\dots,s_t\in \{-1,1\}^d$ be the column vectors of such a matrix. Set $N=2^{t}$, and let $y_1,\dots,y_{N}$ be the enumeration of the elements of $\{0,1\}^{t}$ in the colexicographic order. Then, for $k=1,\dots,N$, we define the $k$-th column $u_k$ of our integer matrix $A$ as follows:
    $$u_k=\sum_{i=1}^{t}2^{i}\cdot y_k(i)\cdot s_i.$$
    Given $1\leq k<\ell\leq N$, and writing $b=\delta(y_k,y_{\ell})$, we have
    $$u_{\ell}-u_k=\sum_{i=1}^{t}2^{i}\cdot (y_{\ell}(i)-y_k(i))\cdot s_i=2^{b}\cdot s_b+\sum_{i=1}^{b-1}2^i\cdot (y_\ell(i)-y_k(i))\cdot s_i.$$
    Here, we used that  $y_1,\dots,y_k$ are ordered colexicographically, ensuring that $y_{\ell}(b)-y_k(b)=1$. From this, it is easy to observe that the sign vector of $u_{\ell}-u_k$ is $s_b$. 

    Now let $R\subset [d]$ and $C\subset [N]$, and let us closely examine what it means for the submatrix of $A$ induced by the rows $R$ and columns $C$ to be row-monotone. Suppose that every row in this submatrix is monotone increasing. Then by the previous observation, we must have that for every $k,\ell\in C$ satisfying $k<\ell$, $s_{b}|_R$ is the all 1 vector, where $b=\delta(y_k,y_{\ell})$. In other words, writing $B\subset [t]$ for the set of indices $b$ such that $s_b|_R$ is the all 1 vector, we have $\delta(y_k,y_{\ell})\in B$ for every $k,\ell\in C$, $k\neq \ell$. But then $|C|\leq 2^{|B|}$, as among any $2^{|B|}+1$ vectors in $\{0,1\}^{t}$, one can find two that agrees on every coordinate in $B$. 
    
    Now suppose that $|R|=n$. By the choice of $s_1,\dots,s_t$, there are less than $\log_2 n$ among these vectors whose restriction to $R$ is the all 1 vector. Hence, $|B|< \log_2 n$, and thus $|C|< n$. To summarize, this means that $A$ contains no $n\times n$ submatrix, in which all rows are monotone increasing. We can proceed similarly in the monotone decreasing case, finishing the proof.
\end{proof}

Let us highlight the following property of the construction presented in the previous proof. The relative order of any two column vectors $u_k$ and $u_{\ell}$ only depends $\delta(y_k,y_{\ell})$. One might visualize this with the help of a binary tree of height $t$, whose leaf labeled with $k$ corresponds to the vector $u_k$. Then the relative order of $u_k$ and $u_{\ell}$ only depends on the common ancestor of $k$ and $\ell$. In order to prove our upper bound, we show that every sufficiently long sequence of vectors contains a long subsequence with this property. 

 \section{Upper bound for row-monotone matrices  --- Proof of Theorem \ref{thm:main_upperbound}}\label{sect:2}

In this section, we prove Theorem \ref{thm:main_upperbound}. Given a matrix, we consider the sequence of its columns vectors. First, with the help of a series of lemmas, we trim this sequence to get a subsequence which highly resembles our construction presented in the proof of Theorem \ref{thm:main_lowerbound}. 

Let us start with some definitions. With slight abuse of notation, we view a sequence  $v_1,\dots,v_N$ as a set $\{(v_i,i):i\in [N]\}$, so set operations $\cup$ and $\cap$ on subsequences are treated accordingly. Given a sequence of vectors $v_1,\dots,v_N\in \mathbb{R}^d$ and $s\in\{-,+\}^d$, we write $v_i\prec_s v_j$ if $i<j$ and $\mbox{sign}(v_j-v_i)=s$. Also, given two subsequences $V$ and $W$, write $V\prec_s W$ if $v\prec_s w$ for every $v\in V,w\in W$. 

\begin{lemma}\label{lemma:bipartite}
Let $v_1,\dots,v_N\in \mathbb{R}^d$ such that $v_i(a)\neq v_j(a)$ for every $a\in [d]$ and $1\leq i<j\leq N$. Then there exists $s\in\{-,+\}^d$ and two subsequences $A$ and $B$ such that $|A|=|B|\geq N/2^{d+1}$, and $A\prec_s B$. 
\end{lemma}

\begin{proof}
We prove the following statement by induction: For every $0 \leq t \leq d$ there exists $s \in \{-,+\}^t$ and two subsequences $A$ and $B$ such that $|A| = |B| \geq \frac{N}{2^{t+1}}$, and $A_t\prec_{s} B_t$. Here for a set $C$, we denote the restriction of the vectors of $C$ onto their first $t$ coordinates by $C_t$. Then the statement of the lemma follows by taking $t = d$. 

We prove this statement by induction on $t$. Our base case is $t=0$, in which case we take $A=\{v_1,\dots,v_{N/2}\}$ and  $B=\{v_{N/2+1},\dots,v_N\}$. The second condition is automatically satisfied, as there are no coordinates to be compared. 

Suppose that $t\geq 1$. Let $s'\in\{-,+\}^{t-1}$ and $A',B'$ be two subsequences such that $|A'|=|B'|\geq N/2^t$ and $A'_{t-1}\prec_{s'} B'_{t-1}$. Order the vectors of $A'\cup B'$ according to their $t$-th coordinate in a monotone increasing manner. Let $X$ be the first half of the vectors in this order, and let $Y$ be the set of the second half. As $|A'\cap X|+|A'\cap Y|=|A'|\geq N/2^t$,  one of $|A'\cap X|$ and  $|A'\cap Y|$ is at least $N/2^{t+1}$. Assume that $|A'\cap X|\geq N/2^{t+1}$, since the other case can be handled similarly. Set $A=A'\cap X$ and $B=B'\cap Y$. Then $|A|=|B|\geq N/2^{t+1}$. Let $s\in \{-,+\}^{t}$ be the vector that agrees with $s'$ on the first $t-1$ coordinates, and $s(t)=+$. Then $A_t\prec_s B_t$, finishing the proof.
\end{proof}

Next, we will apply Lemma \ref{lemma:bipartite} repeatedly to get a long subsequence with certain special structure. We describe this structure with the help of the following definitions.

\begin{definition}
    A \emph{$d$-vector-labeled} (or simply \emph{vector-labeled}) rooted tree is a pair $(T,\lambda)$, where $T$ is a tree and $\lambda$ assigns an element of $\{-,+\}^d$ to each non-leaf vertex of $T$.
\end{definition}

 Recall that $BT_m$ denotes the perfect rooted binary tree of depth $m$, and $\delta(v,w)$ is the common ancestor of $v$ and $w$.

\begin{definition}
    A sequence of vectors $v_1,\dots,v_{2^m}\in \mathbb{R}^d$ is \emph{binary-tree-like} if there exists a $d$-vector-labeled binary tree $(BT_m,\lambda)$ such that for any $1\leq i<j\leq 2^m$, $\lambda(\delta(i,j))=\mbox{sign}(v_j-v_i)$. Say that $(BT_m,\lambda)$ is \emph{associated} with the sequence.
\end{definition}

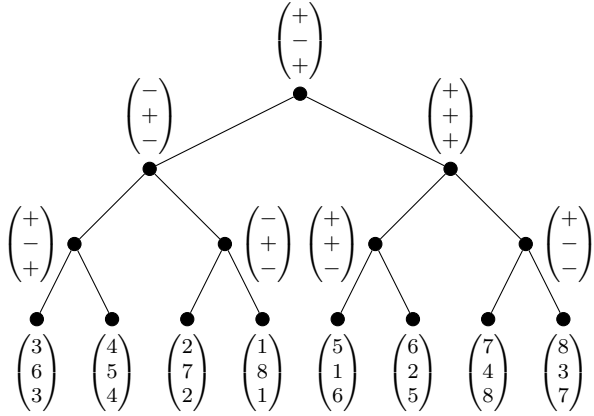
\begin{figure}
\begin{center}
    \begin{tikzpicture}
    \begin{scriptsize}
        \node[vertex,label=below:$\begin{pmatrix} 3 \\ 6 \\ 3 \end{pmatrix}$] (v1) at (-3.5,0) {};
        \node[vertex,label=below:$\begin{pmatrix} 4 \\ 5 \\ 4 \end{pmatrix}$] (v2) at (-2.5,0) {};
        \node[vertex,label=below:$\begin{pmatrix} 2 \\ 7 \\ 2 \end{pmatrix}$] (v3) at (-1.5,0) {};
        \node[vertex,label=below:$\begin{pmatrix} 1 \\ 8 \\ 1 \end{pmatrix}$] (v4) at (-0.5,0) {};
        \node[vertex,label=below:$\begin{pmatrix} 5 \\ 1 \\ 6 \end{pmatrix}$] (v5) at (0.5,0) {};
        \node[vertex,label=below:$\begin{pmatrix} 6 \\ 2 \\ 5 \end{pmatrix}$] (v6) at (1.5,0) {};
        \node[vertex,label=below:$\begin{pmatrix} 7 \\ 4 \\ 8 \end{pmatrix}$] (v7) at (2.5,0) {};
        \node[vertex,label=below:$\begin{pmatrix} 8 \\ 3 \\ 7 \end{pmatrix}$] (v8) at (3.5,0) {};

        \node[vertex,label=left:$\begin{pmatrix} + \\ - \\ + \end{pmatrix}$] (v12) at (-3,1) {};
        \node[vertex,label=right:$\begin{pmatrix} - \\ + \\ - \end{pmatrix}$] (v34) at (-1,1) {};
        \node[vertex,label=left:$\begin{pmatrix} + \\ + \\ - \end{pmatrix}$] (v56) at (1,1) {};
        \node[vertex,label=right:$\begin{pmatrix} + \\ - \\ - \end{pmatrix}$] (v78) at (3,1) {};

        \node[vertex,label=above:$\begin{pmatrix} - \\ + \\ - \end{pmatrix}$] (w1) at (-2,2) {};
        \node[vertex,label=above:$\begin{pmatrix} + \\ + \\ + \end{pmatrix}$] (w2) at (2,2) {};

        \node[vertex,label=above:$\begin{pmatrix} + \\ - \\ + \end{pmatrix}$] (w0) at (0,3) {};

        \draw (v1) -- (v12) -- (w1) -- (w0) ;
        \draw (v2) -- (v12) ;
        \draw (v3) -- (v34) -- (w1) ;
        \draw (v4) -- (v34) ;

        \draw (v5) -- (v56) -- (w2) -- (w0) ;
        \draw (v6) -- (v56) ;
        \draw (v7) -- (v78) -- (w2) ;
        \draw (v8) -- (v78) ;
    \end{scriptsize}
    \end{tikzpicture}
\end{center}
\caption{An illustration of a binary-tree-like sequence of 3-dimensional vectors with the associated vector-labeled binary tree.}  
\label{figure:1}
\end{figure}

See Figure \ref{figure:1} for an illustration of this definition. Note that if $v_1,\dots,v_{2^m}$ is binary-tree-like, then $(BT_m,\lambda)$ is unique, and every $d$-vector-labeled binary tree is associated with a unique binary-tree-like sequence up to order isomorphism. The next is our first key lemma, which tells us that we can find a long binary tree like subsequence in every sequence of vectors.

\begin{lemma}\label{lemma:tree_subsequence}
  Let $N\geq 2^{m(d+1)}$ and $v_1,\dots,v_N\in \mathbb{R}^d$ such that $v_i(a)\neq v_j(a)$ for every $a\in [d]$ and $1\leq i<j\leq N$. Then $v_1,\dots,v_N\in \mathbb{R}^d$ contains a binary-tree-like subsequence of length $2^{m}$.
\end{lemma}

\begin{proof}
    We prove the following stronger statement by induction on $m$. There exists a vector-labeled binary tree $(BT_m,\lambda)$ and $2^m$ subsequences $A_1,\dots,A_{2^m}$ such that $|A_1|=\dots=|A_{2^m}|\geq N/2^{m(d+1)}$, and $A_i\prec_s A_j$ for every $1\leq i<j\leq 2^m$, where $s=\lambda(\delta(i,j))$.
    
    In case $m=1$, apply Lemma \ref{lemma:bipartite} to get $s\in\{-,+\}^d$ and $A_1, A_2$ such that $|A_1|=|A_2|\geq N/2^{d+1}$, and $A_1\prec_s A_2$. Then $(BT_1,\lambda)$ suffices, where $\lambda$ assigns $s$ to the root. Now suppose that $m\geq 2$, and we are given $(BT_{m-1},\lambda')$ and $A_1',\dots,A_{2^{m-1}}'$ satisfying the desired properties. For each $A_i'$, apply Lemma \ref{lemma:bipartite} to get $s_i\in\{-,+\}^d$ and two subsequences $A_{2i-1}$ and $A_{2i}$ such that $$|A_{2i-1}|=|A_{2i}|\geq \frac{|A_i'|}{2^{d+1}}\geq \frac{N}{2^{m(d+1)}}$$ and $A_{2i-1}\prec_{s_i} A_{2i}$. Define the labeling $\lambda$ on $BT_m$ such that every node of depth at most $m-2$ inherits the label from $\lambda'$, and $\delta(2i-1,2i)$ receives the label $s_i$. Then $A_1,\dots,A_{2^m}$ and $(BT_m,\lambda)$ suffices.
\end{proof}

In what follows, we find an even more structured subsequence in a binary-tree like sequence, where we require that vertices of the same depth in the associated tree have the same label as well. Given a vector-labeling $\lambda$ on $T$, we write $(T[S],\lambda)$ instead of $(T[S],\lambda|_{V(T[S])})$ for ease of notation.

\begin{definition}
A vector-labeled tree $(T,\lambda)$ is \emph{layered}, if the label of each non-leaf vertex only depends on its depth. Furthermore, given a subset of leaves $S\subset [2^{m}]$ in the vector-labeled binary tree $(BT_m,\lambda)$, we say that $S$ is \emph{perfect} if $BT_m[S]$ is a perfect rooted binary tree, and $(BT_m[S],\lambda)$ is layered. 
\end{definition}

Next, we prove that every vector-labeled binary tree contains a large perfect set of leaves.

\begin{lemma}\label{lemma:perfect}
    Let $m,t,d$ be positive integers such that $2^m\geq (2^{d+1}m)^t$, and let $(BT_m,\lambda)$ be $d$-vector-labeled. Then there exists a perfect $S\subset [2^m]$ of size $2^t$.  
\end{lemma}

\begin{proof}
We prove the following stronger statement by induction. For $k=0,\dots,t$, there exist $S_1^{(k)},\dots,S_{z_k}^{(k)}\subset [2^m]$  such that
\begin{itemize}
    \item[(i)] $z_k\geq 2^m/(2^{d+1}m)^k$,
    \item[(ii)] $|S_i^{(k)}|=2^{k}$,
    \item[(iii)] $S_i^{(k)}$ is perfect,
    \item[(iv)] for every $1\leq i\leq j\leq z_k$, if a vertex $v$ has the same depth in $BT_m[S_i^{(k)}]$ as a vertex $w$ in $BT_m[S_j^{(k)}]$, then $v$ and $w$ has the same depth and label in $(BT_m,\lambda)$.
\end{itemize}
If this is true, we can conclude the proof by taking $S=S_1^{(t)}$.

In case $k=0$, let $S_i^{(0)}=\{i\}$ for $i=1,\dots,2^m$, then the desired conditions are satisfied. Suppose that $k\geq 1$ and that $S_i^{(k-1)}$ is given for $i=1,\dots,z_{k-1}$ satisfying (i)-(iv). Let $r_i$ denote the root of $S_i^{(k-1)}$, and for $j=1,\dots,z_{k-1}/2$, let $q_j=\delta(r_{2j-1},r_{2j})$. Here, we remark that $r_1,\dots,r_{z_{k-1}}$ are pairwise distinct, and also $q_1,\dots,q_{z_{k-1}/2}$ are pairwise distinct by Claim \ref{claim:distinct_root}. Furthermore, $r_1,\dots,r_{z_{k-1}}$ are pairwise not related as they have the same depth in $BT_m$.

As there are $2^{d}$ possible labels, we can choose $J\subset [z_{k-1}/2]$ of size at least $z_{k-1}/(m2^{d+1})$ such that every $q_j$ for $j\in J$ has the same depth and label in $(BT_m,\lambda)$. Writing $z_k:=|J|\geq z_{k-1}/(m2^{d+1})\geq 2^m/(2^{d+1}m)^k$ and enumerating the elements of $J$ as $j_1<\dots<j_{z_k}$, the sets $$S_{i}^{(k)}=S_{2j_i-1}^{(k-1)}\cup S_{2j_i}^{(k-1)}$$ satisfy the properties (i)-(iv). Here, we used Claim \ref{claim:union} to conclude that $BT_m[S_{i}^{(k)}]$ is a binary tree of height $k$ as the roots of $S_{2j_i-1}^{(k-1)}$ and $S_{2j_i}^{(k-1)}$ are not related (see also the remark after Claim \ref{claim:distinct_root}).
\end{proof}

Now everything is set to prove the main theorem of this section.

\begin{proof}[Proof of Theorem \ref{thm:main_upperbound}]
We show that $c=1000$ suffices. Let $d=8n^2$, $s=\lceil \log_2 n\rceil$, $t=4s^2$, $m=2dt\geq 32n^2(\log_2 n)^2$, and $N\geq 2^{cn^4(\log_2 n)^2}$. Consider a $d\times N$ matrix $A$ given by the sequence of its column vectors $v_1,\dots,v_N\in \mathbb{R}^d$. Without loss of generality, we may assume that $v_i(a)\neq v_j(a)$ for every $1\leq i<j\leq N$, $a\in [d]$, otherwise we may apply some small perturbation, and note that any row-monotone submatrix in the perturbed matrix is also row-monotone in the original. 

As $N\geq 2^{1000n^4(\log_2 n)^2}>2^{m(d+1)}$, we can apply Lemma \ref{lemma:tree_subsequence} to find a binary-tree-like subsequence $u_1,\dots,u_{2^m}$ with associated $d$-vector-labeled binary tree $(BT_m,\lambda)$. Our goal is to construct $D\subset [d]$ and $Q\subset [2^m]$ such that the sequence $(u_i|D:i\in Q)$ is either monotone increasing, or monotone decreasing in every coordinate. This corresponds to an $n\times n$ row-monotone submatrix of $A$.

First of all, we have $2^{d+1}m\leq 2^{2d}$ assuming $n$ is sufficiently large, so the inequality $2^m=2^{2dt}\geq (2^{d+1}m)^t$ is satisfied. Hence, by Lemma \ref{lemma:perfect}, we can find $S\subset [2^m]$ of size $2^t$ such that $S$ is perfect in $(BT_m,\lambda)$. Let $T_0=BT_m[S]$, then $(T_0,\lambda)$ is a layered perfect rooted binary tree of height $t$. Therefore, there exists $w_0,\dots,w_{t-1}\in \{-,+\}^d$ such that the label of a vertex of $T_0$ of depth $i$ is $w_i$ for $i=0,\dots,t-1$. 

Consider the $d\times t$ sized matrix with columns $w_0,\dots,w_{t-1}$. As $t\geq 4s^2$ and $d\geq 4n \cdot 2^s$, we can apply Lemma \ref{lemma:zarankiewicz} to find $D\subset [d]$ of size $n$ and  $Z\subset \{0,\dots,t-1\}$ of size $s$ such that $w_i(j)$ is the same for every $(i,j)\in Z\times D$. Without loss of generality, we may assume that $w_i(j)=+$ for every $(i,j)\in Z\times D$.

Finally, apply Claim \ref{claim:levels} to find $Q\subset S$ of size $2^s\geq n$ such that $T_1:=T_0[Q]=BT_m[Q]$ is a binary tree, and the depth of every vertex of $T_1$ with respect to $T_0$ is in $Z$. Then the label of every non-leaf vertex of $(T_1,\lambda)$ is an element of $\{w_i:i\in Z\}$. In particular, every label appearing in $(T_1,\lambda)$ is $+$ on every coordinate in $D$. But this means that the sequence $(u_i|_D:i\in Q)$ is monotone increasing in every coordinate, as by definition, the sign vector of $u_j-u_i$ is $\lambda(\delta(i,j))$ for $i<j$. This finishes the proof.
    
\end{proof}

\section{Monotone matrices --- Proof of Theorem \ref{thm:main}}\label{sect:3}

In this section, we present the proof of Theorem \ref{thm:main} following the ideas of \cite{BST,GKS}. In particular, we prove the following stronger result.

\begin{theorem}
    There exists a constant $c_0>0$ such that every $64n^4\times 2^{c_0n^4(\log n)^2}$ matrix contains an $n\times n$ monotone submatrix.
\end{theorem}

\begin{proof}
    We may assume that $n$ is sufficiently large by increasing $c_0$. Let $c>0$ be the constant given by Theorem \ref{thm:main_upperbound}, then we show that $c_0=2c$ suffices. Let $d=64n^4$, $N=2^{c_0n^4(\log n)^2}$, and let $A$ be a $d\times N$ matrix. By the Erd\H{o}s-Szekeres theorem, every column contains a monotone subsequence of length $\ell=8n^2$. Without loss of generality, we may assume that at least half of the columns contain a monotone increasing subsequence of such length. Furthermore, by the pigeonhole principle, we can find $N/2\binom{d}{\ell}$ columns such that this monotone increasing subsequence is contained in the same set $R$ of rows. Let $C$ be such a set of columns, and consider the submatrix of $A'$ of $A$ induced by the rows $R$ and columns $C$.  Here, $\binom{d}{\ell}\leq d^{\ell}=n^{O(n^2)}$, and so $|C|\geq N/2\binom{d}{\ell}>2^{cn^4(\log n)^2}$ (assuming $n$ is sufficiently large). Therefore, we can apply Theorem \ref{thm:main_upperbound} to find an $n\times n$ row-monotone submatrix in $A'$. As every column of $A'$ is monotone increasing, this also gives an $n\times n$ monotone submatrix, finishing the proof.
\end{proof}

\section*{Acknowledgments}

RAC is supported by the Swedish Research Council grant VR 2021-03687.\\
ZJ is supported by the SNSF grant 200021\_196965.\\
ER is supported by postdoctoral grant 213-0204 from the Olle Engkvist Foundation.



\begin{thebibliography}{99}

\bibitem{BK10}
M. Bodirsky, and J. K\'ara, 
The complexity of temporal constraint satisfaction problems, J. ACM 57 (2010), Article No. 9.

\bibitem{BST}
M. Buci\'c, B. Sudakov, and T. Tran, Erd\H{o}s-Szekeres theorem for multidimensional arrays, to appear in JEMS, arXiv:1910.13318

\bibitem{BM1}
H. Burkill, and L. Mirsky, Monotonicity, J. Math. Anal. Appl. 41 (1973), 391--410.

\bibitem{BM2}
H. Burkill, and L. Mirsky, Combinatorial problems on the existence of large submatrices I, Discrete Math. 6 (1973), 15--28.

\bibitem{ESz35}
P. Erd\H{o}s, and  G. Szekeres, A combinatorial problem in geometry, Compos. Math. 2 (1935), 463--470.

\bibitem{FFT99}
S. Felsner, P. C. Fishburn, and W. T. Trotter, Finite three dimensional partial orders which are not sphere orders, Discrete Math. 201 (1999), 101--132.

\bibitem{FG93}
P. C. Fishburn, and R. L. Graham, Lexicographic Ramsey Theory, J. Comb. Theory Ser. A 62 (1993), 280--298.

\bibitem{GKS}
A. Gir\~ao, G. Kronenberg, and A. Scott, A multidimensional Ramsey theorem, preprint (2022), arXiv:2210.09227

\bibitem{K73}
K. Kalmanson, On a theorem of Erd\H{o}s and Szekeres, J. Comb. Theory Set. A 15 (1973), 343--346.

\bibitem{K53}
J. B. Kruskal, Monotone subsequences, Proc. Amer. Math. Soc. 4 (1953), 264--274.

\bibitem{L21}
L. Lichev, A note on the Erd\H{o}s-Szekeres theorem in two dimensions,
Electronic Journal of Combinatorics 28 (2) (2021), \#P2.23.

\bibitem{LS18}
N. Linial, and M. Simkin, Monotone subsequences in high-dimensional permutations, Comb. Prob. Comput. 27 (2018), 69--83.

\bibitem{M69}
A. P. Morse, Subfunction structure, Proc. Amer. Math. Soc. 21 (1969), 321--323. 

\bibitem{Steele}
J. M. Steele, Variations on the monotone subsequence theme of Erd\H{o}s and Szekeres. In discrete probability and algorithms  1995 (pp. 111-131). Springer, New York, NY.

\bibitem{SzT01}
T. Szab\'o, and G. Tardos, A multidimensional generalization of the Erd\H{o}s-Szekeres lemma on monotone subsequences, Comb. Prob. Comput. 10 (2001), 557--565.

\end{thebibliography}
\end{document}